\documentclass[12pt]{amsart}
\usepackage[latin1]{inputenc}
\usepackage[english]{babel}

\usepackage{color}

\usepackage{indentfirst}
\usepackage{amssymb}
\usepackage{amsthm}
\usepackage{enumerate}

\newcommand{\vf}{\varphi}
\newcommand{\sm}{\setminus}
\newcommand{\sub}{\subset}

\def\cA{{{\mathcal A}}}

\def\cF{{{\mathcal F}}}
\def\cG{{{\mathcal G}}}

\def\cH{{{\mathcal H}}}

\def\cY{{{\mathcal Y}}}
\def\cZ{{{\mathcal Z}}}

\newcommand{\graph}{\rm gr}
\newcommand{\qu}{\mathbb{Q}}

\newcommand{\btu}{\bigtriangleup}

\newcommand{\er}{\mathbb R}

\newcommand{\la}{\langle}
\newcommand{\ra}{\rangle}

\newcommand{\clop}{\protect{\rm Clop} }
\newtheorem{theo}{Theorem}[section]
\newtheorem{lem}[theo]{Lemma}%[section]
\newtheorem{cor}[theo]{Corollary}%[theo]
\newtheorem{thm}[theo]{Theorem}
\newtheorem{prop}[theo]{Proposition}%[section]
\newtheorem{defn}[theo]{Definition}%[section]
\newtheorem{prob}[theo]{Problem}
\newtheorem{lemma}[theo]{Lemma}

\theoremstyle{definition}

\newtheorem{example}[theo]{Example}

\theoremstyle{remark}
\newtheorem{remark}[theo]{Remark}
\numberwithin{equation}{section}

\def\epsilon{\varepsilon}

\newcommand{\To}{\longrightarrow}

\providecommand{\MR}{\relax\ifhmode\unskip\space\fi MR }
\providecommand{\href}[2]{#2}

\title[Disconnected images of connected spaces]{On disconnected  images of connected spaces}

\author[A.\ Avil\'es]{ Antonio Avil\'{e}s}
\address{Departamento de Matem\'{a}ticas\\
Facultad de Matem\'{a}ticas\\ Universidad de Murcia\\ 30100 Espinardo (Murcia)\\
Spain} \email{avileslo@um.es}

\author[G.\ Plebanek]{Grzegorz Plebanek}
\address{Instytut Matematyczny\\ Uniwersytet Wroc\l awski\\ Pl.\ Grunwaldzki 2/4\\
50-384 Wroc\-\l aw\\ Poland} \email{grzes@math.uni.wroc.pl}

\thanks{The first author supported by MINECO and FEDER (MTM2014-54182-P). Both authors supported by Fundaci\'{o}n S\'{e}neca - Regi\'{o}n de Murcia (19275/PI/14). The second author supported  by the Polish National Science Center research grant NCN grant 2013/11/B/ST1/03596 (2014-2017).}


\begin{document}

\begin{abstract}
We introduce the notion that a zero-dimensional compact space $L$ is a \emph{Boolean image} of an arbitrary compact space $K$. When $K$ is also zero-dimensional, this just means that $L$ is a continuous image of $K$. However,  a number of interesting questions arise when we consider connected compacta $K$.
\end{abstract}

\maketitle

\section{Introduction}

A {\em continuous} image of a connected space is clearly connected but we shall consider here the concept of a {\em Boolean image} that allows to loose connectedness.
 A pseudoclopen in a compact space $K$ is a pair $a = (a^-,a^+)$ such that $a^-$ is closed subset of $K$, $a^+$ is an open subset of $K$ and $a^-\subset a^+$. Pseudoclopens are also called \emph{regular pairs} and serve as a substitute of clopen sets when we work in spaces which are not totally disconnected. If $c$ is a clopen set, then $(c,c)$ is a pseudoclopen, and when we are in a zero-dimensional compact space this is, in a sense, the end of the story because given any pseudoclopen $a$ there exists a clopen $c$ such that $a^- \subset c \subset a^+$. We may think of a a pseudoclopen $a$ as an approximation to an imaginary clopen $c$ for which the only thing we know is that $a^- \subset c \subset a^+$. Even if we are not sure what these imaginary clopens are, we can be sure about some statements that we can make about them. Given pseudoclopens $a_1,\ldots,a_n$ and a logical formula $F(a_1,\ldots,a_n)$, the statement
 $${}^\star [F(a_1,\ldots,a_n)]$$
 means that the statement $F(c_1,\ldots,c_n)$ would hold for any arbitrary sets such that $a_i^- \subset c_i \subset a_i^+$. For example, ${}^\star [a\subset b]$ means that $a^+ \subset b^-$, while ${}^\star [a\neq b]$ means that either $a^-\setminus b^+ \neq\emptyset$ or $b^-\setminus a^+ \neq \emptyset$.  Our work has been motivated by the following result of Brech and Koszmider, which was used in \cite{BreKos} as  an instrumental tool to provide some examples in the theory of Banach spaces of continuous functions.

\begin{thm}[Brech and Koszmider] \label{BK}
If $K$ is  compact convex subset of $\mathbb{R}^\Gamma$, then it is impossible to find pseudoclopens $\{a_\alpha : \alpha<\omega_1\}$ in $K$ such that $^\star[a_\alpha \subset a_\beta]$ whenever $\alpha<\beta$.
\end{thm}

Another result of Koszmider \cite{Kos} can be interpreted in this language:

\begin{thm}[Koszmider] \label{nouncountable}
It is relatively consistent that  there exists a nonmetrizable compact space $K$ which contains no uncountable set of pseudoclopens satifying $^\star [a\neq b]$ whenever $a\neq b$.
\end{thm}

These results suggest a deeper analysis of what one can say about the structure of pseudoclopens in connected compact spaces, that we shall try to initiate in this paper. First, remember that a Boolean polynomial $P(x_1,\ldots,x_n)$ is any algebraic expression of the variables $x_1,\ldots,x_n$ using the Boolean set operations of union, intersection and complementation.

\begin{lem}
	If $a_1,\ldots,a_n$ are pseudoclopens of $K$ and $P(x_1,\ldots,x_n)$ is a Boolean polynomial then $^\star\left[ P(a_1,\ldots,a_n) = \emptyset \right]$ holds if and only if $P(a_1^\pm,\ldots,a_n^\pm) = \emptyset$ holds for any choice of signs.
\end{lem}

\begin{proof}
	For the nonobvious implication, suppose that $p\in P(c_1,\ldots,c_n)$ with $a^-_i\subset c_i \subset a_i^+$. For each $i$ choose sign $\varepsilon_i = +$ if $p\in c_i$ and $\varepsilon_i = -$ if $p\not\in c_i$. Then $p\in P(a_1^{\varepsilon_1},\ldots,a_n^{\varepsilon_n})$, because, the restriction of $c_i$ and $a_i^{\varepsilon_i}$ to the algebra $\{\emptyset,\{p\}\}$ is the same.
\end{proof}

 An isomorphism between two families of pseudoclopens $\cG$ and $\cG'$ is a bijection $\phi:\cG\To \cG'$ such that, for any Boolean polynomial $P(x_1,\ldots,x_n)$ we have that
$${}^\star\left[ P(a_1,\ldots,a_n) = \emptyset\right] \iff {}^\star\left[ P(\phi(a_1),\ldots,\phi(a_n))=\emptyset\right],$$
for any distinct $a_1,\ldots, a_n\in\cG$. For this definition we can restrict just to atomic polynomials $P$ (those who involve only intersection and complementation) because any Boolean polynomial can be written as union of atomic ones.

If $L$ is a compact zero-dimensional space then we write $\clop(L)$ for the algebra of clopen subsets of $L$. We are primarily interested in investigating which families $\cG\sub\clop(L)$ of clopen sets are isomorphic to families of pseudoclopens in another
compact space $K$, which is connected or has some stronger properties.  Our definition of isomorphism in this particular case reads as follows:

\begin{defn}
Given a family of sets $\cG$ and a bijection $\phi$ from $\cG$ onto a family of pseudoclopens of a compact space $K$, we say that $\phi$ is an isomorphism if
for every $m,n$ and any distinct $a_1,\ldots, a_n, b_1,\ldots b_m\in\cG$
$$\mbox{ if } \bigcap_{i\le n} a_i\sm \bigcup_{j\le m} b_j=\emptyset \mbox{ then } \bigcap_{i\le n} \phi(a_i)^+\sm \bigcup_{j\le m} \phi(b_j)^-=\emptyset;$$
$$\mbox{ if }   \bigcap_{i\le n} \phi(a_i)^-\sm \bigcup_{j\le m} \phi(b_j)^+=\emptyset \mbox{ then }
\bigcap_{i\le n} a_i\sm \bigcup_{j\le m} b_j=\emptyset.$$
\end{defn}

If $P(x_1,\ldots, x_n)$ is a Boolean polynomial and $a_1,\ldots, a_n$ are pseudoclopens then we can define a pseudoclopen
$$P(a_1,\ldots, a_n)=\big( P(a_1,\ldots, a_n)^-, P(a_1,\ldots, a_n)^+\big),$$
where $P(a_1,\ldots,a_n)^- = \bigcap P(a_1^\pm,\ldots,a_n^\pm)$ and $P(a_1,\ldots,a_n) = \bigcup P(a_1^\pm,\ldots,a_n^\pm)$ are, respectivelye, the intersection and union of all evaluations for all possible choices of signs. These are the least and the largest value that we could obtain when applying the polyonomial $P$ to our imaginary clopens. We can also compute this through the recursive formulas $(P\wedge Q)^\varepsilon = P^\varepsilon \wedge Q^\varepsilon$, $(P\vee Q)^\varepsilon = P^\varepsilon \vee Q^\varepsilon$, and $(P^c)^\varepsilon = (P^{-\varepsilon})^c$. With this notation at hand, we can express that
$\phi$ in the definition above is an isomorphism if the equivalences
$$P(a_1,\ldots, a_n)=\emptyset \iff P(\phi(a_1),\ldots, \phi(a_n))^+=\emptyset \iff P(\phi(a_1),\ldots, \phi(a_n))^-=\emptyset.$$
hold for any distinct $a_1,\ldots, a_n\in\cG$.

Note that here it is not essential that $\cG$ is a family of clopen sets in a compact space. We can as well consider any family  $\cG$ in some Boolean algebra $\mathfrak A$,
and treat $a\in\cG$ as a corresponding clopen set in the Stone space of $\mathfrak A$.

A subfamily $\mathcal{G}$ of $\clop(L)$  separates the points of $L$ if for every two different points $x,y\in L$ there exist  $a\in L$ such that $|\{x,y\}\cap a| = 1$.
The next definition introduces the concept that is basic for our considerations.

\begin{defn}
We say that a compact zero-dimensional space $L$ is a Boolean image of a compact space $K$
if there is a family of clopens of $L$ that separates the points of $L$ and which is isomorphic
to a family of pseudoclopens of $K$.
\end{defn}

In the sequel, $K$ and $L$ (with possible indices) always denote compact Hausdorff spaces.
Notice that the concept of a Boolean image generalizes the usual notion of a continuous image.

\begin{remark}\label{basic_remark}
Let  the compact space $K$ be zero-dimensional.  Then $L$ is a Boolean image of $K$ if and only if $L$ is a continuous image of $K$.
\end{remark}

\begin{proof}
If $g:K\to L$ is a continuous surjection then clearly the operation $c\to g^{-1}(c)$ is an isomorphism between $\clop(L)$ and a subalgebra of $\clop(K)$.

If $\cG\sub\clop(L)$ separates the points of $L$ and $\vf:\cG\to \cG'$ establishes   an isomorphism with a family $\cG'$ of pseudoclopens in $K$ then, since  $K$ is zero-dimensional,
we can assume that in fact $\cG'\sub\clop(K)$. Given $x\in K$, the family $\{c\in \cG: x\in\vf(c)\}$ is centred and its intersection must be a singleton, say $\{g(x)\}$.
It is easy to check that this defines a continuous surjection $g:K\to L$.
\end{proof}

In a sense, we form the notion of a Boolean image
by twisting the definition of continuous image, so as to possibly have disconnected images of connected spaces.
For instance, Theorem~\ref{nouncountable} can be now restated by saying that, consistently, there exists a compact nonmetrizable space such that all its Boolean images are metrizable. Another result from the literature that can be read in this manner is a theorem on the space
$$B(\Gamma) = \{x\in \mathbb{R}^\Gamma : \sum_{\gamma\in \Gamma}|x_\gamma|\leq 1\},$$
 stating that {\em no product of more than one nonmetrizable spaces can be a Boolean image of $B(\Gamma)$}, see \cite{donotmap}.

One of our main results presented here offers  an essential improvement of Theorem~\ref{BK} and reads as follows.

\begin{thm}\label{alternative}
If $K$ is separably connected and $L$ is a Boolean image of $K$, then either $L$ is Corson compact or $L$ maps continuously onto $2^{\omega_1}$.
\end{thm}

Here we say  that $K$ is {\em separably connected} if every two points of $K$ are contained in a connected separable subspace of $K$.
Note that every path-connected space and hence every convex set in a topological vector space is separably connected.
A compact space $L$ is Corson compact if, for some $\Gamma$, it is homeomorphic to a subset of the $\Sigma$-product
$$\Sigma(\mathbb{R}^\Gamma) = \{x\in\mathbb{R}^\Gamma : |\{\gamma : x_\gamma\neq 0\}| \leq \aleph_0\}.$$
The ordinal interval compact space $[0,\omega_1]$ is neither Corson compact nor does it map onto $2^{\omega_1}$; on the other hand,
 it has an $\omega_1$-chain of clopens that separates its points. Hence, Theorem~\ref{BK} follows from Theorem~\ref{alternative}. This result as well as some more technical results and interesting corollaries about the structure of pseudoclopens in separably connected spaces are presented in Section~\ref{secsepcon}.

In Section~\ref{secconnected} we try to understand which zero-dimensional compact spaces are Boolean images of compact connected spaces.
 The main result is that a family of sets $\cG$  happens to be isomorphic to a family of pseudoclopens of a connected space if and only if for every finite subfamily of $\cG$, the graph of nonempty atomic intersections is connected.
 Consequently, such a family $\cG$ must be strongly irredundant; in particular, no $a\in\cG$ can be generated from $\cG\sm\{a\}$.
 It follows that $F$-spaces and nonmetrizable scattered spaces of height 3 are not Boolean images of connected compacta.
 On the positive side, we prove that every zero-dimensional compact line is a Boolean image of a connected compactum.
 This yields an example of zero-dimensional compact space  that is a Boolean  image of a connected compactum
 but is not a Boolean image of a separably connected compact space.

Finally, in Section~\ref{secconvex} we study Boolean images of convex compacta. Theorem \ref{decconvex} shows that if $L$ is such an image then
$\clop(L)$ has a generating family with special structural properties. We give an example of a zero-dimensional space that is a Boolean image of a path-connected compactum but not of a compact convex set.

We are very grateful to the referee for his/her careful reading and several critical comments.

\section{Some basic facts}

We consider throughout this section  a compact zero-dimensional space $L$ and  an arbitrary compact  space $K$.

\begin{prop}\label{preimage}
If $L$ is a Boolean image of $K$, and $K$ is a continuous image of $K_0$, then $L$ is a Boolean image of $K_0$. In particular, if $L$ is a Boolean image of $K$, then $L$ is a continuous image of every zero-dimensional compact space that maps continuously onto the space $K$.
\end{prop}

\begin{proof}
If $f:K_0\To K$ is the continuous onto map and $\{(a_i^-,a_i^+) : i\in I\}$ is a family of pseudoclopens of $K$ that is isomorphic to a point-separating subfamily of
$\clop(L)$,  then simply consider the family $\{(f^{-1}[a_i^-],f^{-1}[a_i^+]) : i\in I\}$
of pseudoclopens in $K_0$. The second statement follows from Remark \ref{basic_remark}.
\end{proof}

We do not know if the converse of the last assertion holds true, see
Problem \ref{bf:problems}.

\begin{prop}\label{cios}
If $L$ is a Boolean image of $K$, then $L$ is a continuous image of a closed subspace of $K$.
\end{prop}

\begin{proof}
Let $\phi:\mathcal{G}\To \mathcal{F}$ be an isomorphism of a point-separating family $\mathcal{G}\sub\clop(L)$ with a family $\cF$ of pseudoclopens of $K$. Consider $$K_0 = \bigcap_{a\in \mathcal{G}}\left( \phi(a)^- \cup (K\setminus \phi(a)^+)\right),$$
and define $f:K_0\To L$ by declaring $f(x)$ to be the only point in
$$L(x)=\bigcap\{a\in \mathcal{G} : x\in \phi(a)^-\}\setminus \bigcup \{a\in \mathcal{G} : x\not\in \phi(a)^+\}.$$
To check that $f$ is well-defined note first that  we cannot have two different elements in $L(x)$;
indeed, take any $y_1,y_2\in L$, $y_1\neq y_2$. Because the family $\mathcal{G}$ separates the points of $L$
we can assume, by symmetry,  that $y_1\in a$ and $y_2\notin a$ for some $a\in\cG$. Since $x\in K_0$ we have
either $x\in \phi(a)^-$ giving $L(x)\sub a$ and $y_2\notin L(x)$, or $x\notin\phi(a)^+$ giving $L(x)\cap a=\emptyset$ and $y_1\notin L(x)$.

To see that  the  set $L(x)$  must be nonempty, note that the family
$$ \{a\in\cG: x\in \phi(a)^- \mbox{ or } x\notin\phi (a)^+\},$$
has the finite intersection property, as  $\phi$ is an isomorphism. Therefore, by compactness, such a family has a nonempty intersection.

In the same manner we check that for every $y\in L$
$$K(y)=\bigcap\{\phi(a)^- : a\in \mathcal{G}, y\in a\}\setminus \bigcup\{\phi(a)^+ : a\in \mathcal{G}, y\not\in a\}\neq\emptyset.$$
taking $x\in K(y)$ we have $f(x)=y$, so $f$ is a surjection.

Finally, note that if  $a\in \mathcal{G}$ then
 $$f^{-1}[a] = \phi(a)^-\cap K_0=\phi^+(a)\cap K_0$$
 is a clopen subset of $K_0$.
 Since $\cG$ separates the points of $L$, $\cG$ generates $\clop(L)$ and therefore $f^{-1}[b]$ is clopen for every $b\in \clop(L)$.
 Hence $f$ is continuous.
\end{proof}

The way   of reasoning presented above, combining compactness and the notion of isomorphism will be repeated several times along this paper.
It will become clear later that the implication in Proposition \ref{cios} cannot be reversed: every $L$ is a continuous image of a closed subspace
of $[0,1]^\Gamma$ but, as we shall see, some $L$ are not Boolean images of the Tikhonov cubes.

\begin{prop}
If $L_i$ is a Boolean image of $K_i$ for $i\in I$, then $L=\prod_{i\in I} L_i$ is a Boolean image of $K=\prod_{i\in I} K_i$.
\end{prop}

\begin{proof}
For every $i$, consider a point-separating family $\mathcal{G}_i\sub\clop(L_i)$ and a family $\mathcal{F}_i$ of pseudoclopens of $K_i$ that is  isomorphic to
$\cG_i$. Write, for every $i\in I$, $\pi_i:L\To L_i$ and $p_i:K\to K_i$ for the projections, and put
$$\cG=\bigcup_{i\in I} \left\{\pi_i^{-1}[a]: a\in \cG_i\right\},
\quad \cF=\bigcup_{i\in I} \left\{ (p_i^{-1}[b^-], p_i^{-1}[b^+]) : b\in \cF_i\right\}. $$
It is routine  to check that the families are isomorphic.
\end{proof}

\begin{cor}
For every infinite set $\Gamma$ the space $2^\Gamma$ is a Boolean image of $[0,1]^\Gamma$.
\end{cor}

Let us remark that $[0,1]^\Gamma$ is, in a sense, the smallest compact connected space having $2^\Gamma$ as a Boolean image:
If $2^\Gamma$ is a Boolean image of a compact space $K$ then $K$ can be continuously mapped onto
$[0,1]^\Gamma$. This follows from the fact that $2^\omega$ maps continuously onto $[0,1]^\omega$.

There are two natural questions concerning Boolean and continuous images we have not been able to settle.

\begin{prob}\label{bf:problems}
%$\mbox{ }$
\begin{enumerate}[(i)]
\item Suppose that $L$ is a continuous image of every zero-dimensional space that maps continuously onto $K$.  Does this imply that $L$ is a Boolean image of $K$?
\item Let $L$ be a Boolean image of $K$ and suppose that $L'$ is a continuous image of $L$. Is $L'$  a Boolean image of $K$?
\end{enumerate}
\end{prob}

\section{Boolean images of connected spaces}\label{secconnected}

Given a finite family $\mathcal{F}$ of sets, we consider the graph $\graph(\mathcal{F})$ which is defined as follows.
The set of its vertices consists  of  the subsets $\cY\subset \mathcal{F}$  such that
$$\bigcap  \cY \setminus \bigcup \big( \mathcal{F}\setminus \cY\big) \neq\emptyset.$$
We declare that vertices   $\cY,\cZ$ are joined by an edge if and only if their symmetric difference
$\cY \bigtriangleup \cZ = (\cY\setminus \cZ)\cup (\cZ\setminus \cY)$ is exactly one point.

\begin{lem}\label{connected:finite}
Let $I$ be a finite set and let $L\sub 2^I$; write $F_i=\{x\in L: x_i=1\}$ for $i\in I$ and  $\cF=\{F_i:i\in I\}$.
Let $K\sub [0,1]^I$ be set of all $y$ such that for some $k\in I$ there are $x,x'\in L$ such that
$y_j=x_j=x_j'$ for $j\in I\sm\{k\}$ and $x_k'\le y_k\le x_k$.

Then $K$ is a compact subspace of $[0,1]^I$ and $K$ is connected whenever $\graph(\cF)$ is a connected graph. Moreover,  the mapping
$$F_i \mapsto (F_i,V_i ), \mbox{ where } V_i=\{x\in K : x_i>0\},$$
establishes an isomorphism between the family $\cF$ and a family of pseudoclopens of $K$.
\end{lem}

\begin{proof}
Clearly $K$ is compact as it  is a union of $L$ and a finite number of segments contained in $[0,1]^I$.
Note that every $x\in L$ corresponds to the vertex $\cY\sub\cF$, where $\cY=\{F_i\in\cF: x_i=1\}$ and $K$ contains
a segment joining two vertices that are joined by an edge in $\graph(\cF)$. Therefore, if $\graph(\cF)$ is a connected graph then
$K$ is path-connected.

Suppose that $y\in \bigcap_{i\in J}  V_i\sm \bigcup_{i\in I\sm J} F_i\neq\emptyset$ for some $J\sub I$ and consider
the set $$H=\bigcap_{j\in J}  F_i\sm \bigcup_{i\in I\sm J} F_i.$$
Then either $y\in L$ and $y\in H$ or there is $k\in I$ and $x,x'\in L$ with $x_k'=0$, $x_k=1$ and $y_j=x_j=x_j'$ for $i\neq k$.
Then $x\in  H$ if $k\in J$ and $x'\in  H$
otherwise. Hence, in each case, $H\neq\emptyset$.
Since $L\sub K$, this shows that the mapping $F_i\mapsto (F_i,V_i )$ is indeed an isomorphism.
\end{proof}

\begin{thm}\label{connected}
For a family $\mathcal{G}$ of sets  the following are equivalent
\begin{enumerate}[(i)]
\item the family $\mathcal{G}$ is isomorphic to a family of pseudoclopens in a compact connected space;
\item for every nonempty finite $\mathcal{F}\subset\mathcal{G}$, $\graph(\mathcal{F})$ is a connected graph.
\end{enumerate}
\end{thm}

\begin{proof}
We first prove the implication $(ii)\to (i)$. Without loss of generality we can assume that $\cG\sub \clop(L)$, where $L$ is compact and zerodimensional and $\cG$ separates the points of $L$.
In turn,  an application of the embedding $L\ni x\mapsto (1_G(x))_{G\in \cG}\in 2^\cG$ reduces the situation to the case
when $L\sub 2^\Gamma$ and $\cG=\{F_\gamma: \gamma\in\Gamma\}$, where $F_\gamma=\{x\in L: x_\gamma=1\}$.

 For every finite $I\sub \Gamma$, we consider the projection $L_I=\pi_I[L]$
and the family $\cF_I=\{\pi_I[F_\gamma]:\gamma\in I\}$. By Lemma \ref{connected:finite}, $\cF_I$ is isomorphic to a family of pseudoclopens
of a compact connected space $K_I\sub [0,1]^I$. It follows from the construction of $K_I$ that for $I_1\sub I_2\sub \Gamma$,
$K_{I_1}$ is a projection of $K_{I_2}$.

We have defined in this way an inverse system $(K_I)_{I\in [\Gamma]^{<\omega}}$ of compact connected spaces (with projections acting as bonding maps).
It follows that the limit $K$ of that system is compact and connected, see eg,\ \cite{En}, 6.1.18.
It follows easily from the basic properties of inverse systems that  the mapping $F_\gamma\mapsto (F_\gamma, V_\gamma)$, where
$V_\gamma=\{y\in K: y_\gamma>0\}$, gives the required  isomorphism
between $\cF$ and a family of pseudoclopens of $K$.

To verify $(i)\to (ii)$ consider a finite family $\cF$ of sets and an isomorphism $\phi$ from $\cF$  onto a family of pseudoclopens in a compact connected
space $K$.  For any vertex $\cY$ in $\graph(\cF)$ consider the open set $V(\cY)\sub K$, where
$$ V(\cY)=\bigcap_{a\in \cY} \phi(a)^+\sm\bigcup_{a\in\cF\sm\cY} \phi(a)^-.$$

\noindent {\sc Claim.} If $V(\cY)\cap V(\cZ)\neq\emptyset$ then the vertices $\cY,\cZ$ are connected by a path in $\graph(\cF)$.
\medskip

Consider $a\in \cY\sm\cZ$ and set $\cY'=\cY\sm\{a\}$. Take $y\in V(\cY)\cap V(\cZ)$. Then $y\in \phi(b)^+$ for $b\in \cY'$ and
$y\notin \phi(b)^-$ for $b\notin\cY'$, because $a\notin \cZ$.
Therefore $y\in V(\cY')\cap V(\cZ)$.

Since $\phi$ is an isomorphism, $\cY'$ is a vertex in $\graph(\cF)$,  and
$\left| \cY'\btu \cZ\right| =\left| \cY\btu \cZ\right|-1$. Hence, by induction, there is a path in $\graph(\cF)$ leading from $\cY$ to $\cZ$.
\medskip

The space $K$  is a union of $V(\cY)$ over all the vertices $\cY$; indeed, for a given $y\in K$ consider $\cY=\{s\in\cF: y\in\phi(a)^+\}$ which is a vertex
since $\phi$ is an isomorphism.
Therefore, the claim shows that  $\graph(\cF)$ cannot be split into two nonempty sets of vertices that are not connected; otherwise  $K$ would be a disjoint union
of nonempty open sets, which is impossible as $K$ is a connected topological space.
 \end{proof}

The referee has remarked that the connected space $K$ constructed in the first part of Theorem \ref{connected} is of covering dimension one.

\begin{cor}\label{wn:1}
 A family of sets $\mathcal{G}$ is isomorphic to a family of pseudoclopens in a compact connected space if and only if
 every finite $\cF\sub\cG$ has such a property.
\end{cor}

\begin{remark}
Let $A,B,C$ be subsets of a space $X$ such that
$$D=A\cap B=B\cap C=A\cap C.$$
If $D\notin\{\emptyset, X\}$ then the family $\cF=\{A,B,C\}$ is not isomorphic to pseudoclopens in a connected space; indeed,
in $\graph(\cF)$ the vertex $\cF$ is isolated.
\end{remark}

For a family $\cG$ contained in some Boolean algebra we write $\la\cG\ra$ for the subalgebra generated by $\cG$.

\begin{cor}\label{wn:2}
Let $\cG$ be a family of sets such that for every nonempty finite $\cF\sub\cG$ there is $a\in\cF$ containing no nonempty set
from $\la\cF\sm\{a\}\ra$. Then $\cG$ is isomorphic to a family of pseudoclopens in a compact connected space.
\end{cor}

\begin{proof}
We shall check that the graph $\graph(\cF)$ is connected for every nonempty finite $\cF\sub\cG$. This can be done by induction on
$|\cF|$.  Take $a\in\cF$ that contains no nonempty element from $\la\cF\sm\{a\}\ra$; let $\cF'=\cF\sm\{a\}$.

Consider any vertices $\cY,\cZ$ in $\graph(\cF)$ and the following three cases.

\begin{enumerate}[(i)]
\item $a\notin\cY$, $a\notin\cZ$. Then $\cY,\cZ$ may be seen as verices in $\graph(\cF')$; moreover every vertex in $\graph(\cF')$ is also a vertex in
$\graph(\cF)$ (by the property of $a$) so $\cY,\cZ$ are connected by a path by induction.
\item $a\in\cY$, $a\in\cZ$. Put $\cY'=\cY\sm\{a\}$ and  $\cZ'=\cZ\sm\{a\}$. Then $\cY'$ is a vertex in $\graph(\cF)$ joined with $\cY$ by an edge. Moreover,
there is a path form $\cY'$ to $\cZ'$ in $\graph(\cF')$ so there is a path from $\cY$ to $\cZ$ in $\graph(\cF)$.
\item $a\in\cY$, $a\notin\cZ$. Consider $\cY'$ and $\cZ$ and argue as above.
\end{enumerate}

Now we finish the proof applying  Theorem \ref{connected}.
\end{proof}

\begin{cor}\label{wn:3}
Let $\cG$ be a family of sets such that every $a,b\in\cG$ are either comparable by inclusion or disjoint.
Then there  is $\cG'\sub\cG$ such that $\la\cG'\ra=\la\cG\ra$ and
$\cG'$ is isomorphic to a family of pseudoclopens in a compact connected space.
\end{cor}

\begin{proof}
We may suppose that every $a\in\cG$ is nonempty.
Let $\cG'$ be a subfamily of $\cG$ that is maximal with respect to the property
$$\mbox{ for every } a,b_1,\ldots, b_n\in  \cG', \mbox{ if } b_1\cup\ldots b_n\sub a \mbox{ then } a\sm\bigcup_{i\le n} b_i\notin\cG'.$$

Note first that $\la\cG'\ra=\la \cG\ra$. Indeed if $x\in\cG\sm\cG'$ then $\cG'\cup\{x\}$ does not have the above property so
there are $a,b_1,\ldots, b_n\in\cG'\cup \{x\}$, $a$ containing all $b_i$, such that
$$a\sm\bigcup_{i\le n} b_i=g\in\cG'\cup\{x\}.$$
If $a=x$ then necessarily $g\neq x$ and $b_i\neq x$ (because $\emptyset\notin\cG$) so $b_i\in\cG'$ for every $i$;  this gives $x=\bigcup_i b_i \cup g\in \la\cG'\ra$.
The remaining case, when $x=b_i$ for some $i$, is similar.

If $\cF\sub\cG'$ is finite and nonempty then take $a\in\cF$ which is minimal in $\cF$ with respect to inclusion. Then from the property of $\cG'$ it follows that $a$ contains no nonempty element from
$\la\cF\sm\{a\}\ra$ and we may apply Corollary \ref{wn:2}
\end{proof}

\begin{cor}\label{wn:4}
Let $L$ be a zero-dimensional compact space which is a continuous image of a zero-dimensional compact line $L^*$. Then
$L$ is a Boolean image of a compact connected space.
\end{cor}

\begin{proof}
Since $L$ is a continuous image of $L^*$, $\clop(L)$ may be seen as a subalgebra of $\clop(L^*)$. This implies  that
$\clop(L)$ is generated by a family $\cG$ such that every $a,b\in\cG$ are either comparable or disjoint, see Heindorf \cite{He97}.
Hence we finish the proof applying Corollary \ref{wn:3}.
\end{proof}

In connection with the above result, it is worth recalling that the class of continuous images of compact lines coincides with the
class of monotonically normal compact spaces; this highly nontrivial result is due to Mary Ellen Rudin who proved a long-standing conjecture, see \cite{Ru01}.

Let us recall that a $\pi$-base of a topological space is a collection $\mathcal{B}$ of non-empty open sets such that every non-empty open set contains an element of $\mathcal{B}$. The $\pi$-weight of the space is the minimal cardinality of a $\pi$-basis. This cardinal invariant lies between the density character and the weight of the space.

\begin{cor}\label{wn:5}
For every zero-dimensional compact space $L$ there exists a zero-dimensional continuous image $L_1$ of $L$ whose weight equals the $\pi$-weight of $L$ such that $L_1$ is the Boolean image of a connected compact space.
\end{cor}

\begin{proof}
Let $\kappa$ be the $\pi$-weight of $L$. We construct a transfinite sequence $\{a_\alpha : \alpha<\kappa\}$ of clopen subsets of $L$ such that, for every $\alpha<\kappa$, $a_\alpha$ does not contain any non-empty element of the algebra generated by  $\{a_\beta : \beta<\alpha\}$. This can be done by induction since we know that $\{a_\beta : \beta<\alpha\}$ cannot be a $\pi$-basis, as its cardinality is less than the $\pi$-weight of $L$.

 Then the family $\mathcal{G} = \{a_\alpha : \alpha<\kappa\}$ satisfies the assumption of  Corollary \ref{wn:2},
 so $\cG$ is isomorphic to a family of pseudoclopen of a connected compactum.
 Hence  the Stone space $L_1$ of the algebra generated by $\cG$ has the required property.
\end{proof}

Recall that a family $\mathcal{G}$ of elements of a Boolean algebra is {\em irredundant} if  $a\notin \la\mathcal{G}\sm\{a\}\ra$
for every $a\in\cG$.
We shall say that $\mathcal{G}$ is {\em strongly irredundant} if  for every disjoint $\mathcal{F},\mathcal{F}'\subset \mathcal{G}$,
the algebra $\la\cF\ra\cap \la\cF'\ra$ is trivial.

\begin{cor}\label{wn:6}
If $\mathcal{G}$ is isomorphic to a family of pseudoclopens of a connected compactum, then $\mathcal{G}$ is strongly irredundant, so in particular irredundant.
\end{cor}

\begin{proof}
Suppose $x$ is a nonzero element in $\la\cF\ra\cap\la\cF'\ra$, where   $\mathcal{F},\mathcal{F}'\subset \mathcal{G}$ are finite and disjoint.
We consider the graph $G=\graph(\mathcal{F}\cup \mathcal{F}')$.

Let $\cZ$ be a vertex in $G$ such that
$$\bigcap\cZ\sm\bigcup\left( (\cF\cup\cF') \sm\cZ\right)\subset x.$$
Then for $\cY=\cZ\cap\cF$, putting
$y=\bigcap\cY\sm\bigcup(\cF\sm\cY)$,  we have $y\cap x\neq\emptyset$. But $y$ is an atom of $\la\cF\ra$ and $x\in \la\cF\ra$ so $y\sub x$.
The same observation applies to $\cF'$.

It follows that if $\cZ_1$ is another vertex in $G$ and
$\left| \cZ\btu\cZ_1\right|=1$ then again $$\bigcap\cZ_1\sm\bigcup(\cF\cup\cF'\sm\cZ_1)\sub x.$$
Since $G$ is connected by Theorem \ref{connected}, it follows that every atom of $\la\cF\cup\cF'\ra$ is contained in $x$; hence  $x=1$, and the
proof is complete.
\end{proof}

\begin{remark}
One can prove Corollary \ref{wn:6} directly using the notation given in the introductory section. This is an outline of  such a proof: Suppose that
$x\in \la\cF\ra\cap\la\cF'\ra$, where   $\mathcal{F},\mathcal{F}'\subset \mathcal{G}$ are finite and disjoint.
Then $x=P(a_1,\ldots, a_n)=Q(b_1,\ldots, b_k)$ for some Boolean polynomials $P$ and $Q$, and $a_i, b_i\in\cG $ all distinct.
If $\phi$ is an isomorphism between $\cG$ and pseudoclopens in a connected space $K$ then
$$P(\phi(a_1),\ldots)^+\sub Q(\phi(b_1),\ldots)^-\sub P(\phi(a_1),\ldots)^-,$$
which implies that the set $P(\phi(a_1),\ldots)^+=P(\phi(a_1),\ldots)^-$ is clopen so it is  either $\emptyset$ or $K$. It follows that   $x$ is trivial.
\end{remark}

As we have seen, Boolean images $L$ of connected spaces have an irredundant family of clopen sets that separates the points. This easily implies that the algebra $\clop(L)$  has coinitiality  $\aleph_0$, i.e.\ it is a union of increasing countable sequence of its proper subalgebras.
Such a property is not shared by $\clop(L)$ where $L$ is a  compact $F$-spaces, see Geschke \cite{coinitiality}. In particular, neither $\beta\omega$ nor any extremally disconnected compact spaces is a Boolean image of a compact connected space.

\begin{example}
There is a strongly irredundant finite family $\cG$
which is not isomorphic to a family of pseudoclopens in a connected space.

Let $X$ be the set of all $x\in\{0,1\}^5$ such that $|\{i: x_i=0\}|\neq 3$.
Consider $E_i=\{x\in X: x_i=1\}$ and $\cG = \{E_1,E_2,E_3,E_4,E_5\}$.
Then $\graph(\cG)$ is not connected since  every  vertex $\cY\sub\cG$ satisfies either   $|\cY|\ge 3$ or $|\cY|\le 1$
and those two parts of the graph are not connected by an edge.

On the other hand, $\cG$ is strongly irredundant. If not, there exists a partition $\cG = \cF \cup\cH$ , and a nontrivial element $A \in\la\cF\ra\cap \la\cH\ra$.
By symmetry we can assume that $\cF\sub \{E_1,E_2\}$; suppose for instance that
$\cF = \{E_1,E_2\}$, the other cases are similar.
Take $x\in A$ and $y\in X\sm A$.
Then for $x'=(x_1,x_2,1,1,1)$ and $y'=(y_1,y_2,1,1,1)$ we have $x',y'\in X$ and $x'\in A$, $y'\notin A$ because $A$ is determined by the first two coordinates.
But then $A\notin \la\cH\ra$ since $A$ is not determined by the remaining three coordinates.
\end{example}

The next result gives  another example of compact space which is not a Boolean image of a connected compactum.

\begin{prop}
If $L$ is a separable nonmetrizable scattered space of height 3, then there is no uncountable strongly irredundant family of clopen subsets of $L$. Therefore $L$ is not a Boolean image of a connected compact space.
\end{prop}

\begin{proof}
We can suppose that $L''$ the second Cantor-Bendixson derivative of $L$ consists of exactly one point, that we call $\infty$. Suppose that we have an uncountable strongly irredundant family $\mathcal{G}$ of clopen subsets of $L$. By switching some elements to their complement, we can suppose that no clopen in $\mathcal{G}$ contains $\infty$. By passing to an uncountable family we can suppose that there is one element of the countable dense set that belongs to all clopens in $\mathcal{G}$. We can also suppose that $\{a\cap L' : a\in\mathcal{G}\}$ is a $\Delta$-system, since all those sets are finite. Consider a partition $\{a_\gamma,b_\gamma\}_{\gamma\in\Gamma}$  of
 $\mathcal{G}$ into doubletons. All intersections $a_\gamma\cap b_\gamma \cap L'$ are equal to the root of this $\Delta$-system, hence all $a_\gamma\cap b_\gamma$ are equal modulo finite. Hence, since there are only countably many possibilities for $a_\gamma\cap b_\gamma\setminus L'$, there are two such intersections which are equal. They are moreover  nonempty, since we fixed an element of the dense set which is in  all clopens. This implies that $\mathcal{G}$ is not strongly irredundant.
\end{proof}

Let us consider now, for $n=1,2,3,\ldots$ and an uncountable set $\Gamma$, the zero-dimensional compact space
$$\sigma_n(\Gamma) = \left\{x\in 2^\Gamma : |\{\gamma : x_\gamma=1\}| \leq n\right\}.$$
If $n=1$, this is just the one-point compactification of a discrete set of size $|\Gamma|$. It is an easy exercise to check that $\sigma_1(\Gamma)$ is a Boolean image of $K$ if and only if $K$ contains $|\Gamma|$ many pairwise disjoint open sets.
We shall prove below, see Proposition \ref{5.5}, that the spaces $\sigma_n(\Gamma)$  are Boolean images of convex compact spaces. The question about subspaces of $\sigma_n(\Gamma)$ seems to be more delicate. The first nontrivial case is $n=2$,

\begin{prop}
If $L$ is a compact subspace of $\sigma_2(\Gamma)$ then $L$ is a Boolean image of a compact connected space.
\end{prop}

\begin{proof}
Let $\Gamma_0=\{\gamma\in \Gamma: 1_{\{\gamma\}}\in L\}$. Given finite $t\subset \Gamma$, let
$$E_t = \{x \in L : x_\gamma=1 \mbox{ for all } \gamma\in t\};$$
 clearly $E_t$ is a clopen subset of $L$. We write $E_\gamma$ rather than $E_{\{\gamma\}}$.

 Consider the family $\cG$ consisting of all $E_t$ such that either $t=\{\gamma\}$ and  $\gamma\in \Gamma_0$ or $|t|=2$, $1_t \in L$ and $t\not\subset \Gamma_0$.
 Then the family $\cG$ separates the points of $L$. Let us check that  $\cG$ satisfies  condition $(ii)$  of Theorem~\ref{connected}.

 The intersection of three different clopen from $\cG$ is always empty, while if
$a_1,a_2\in\cG$ and $a_1\cap a_2\neq\emptyset$ then

\begin{enumerate}[(i)]
\item either $a_i=E_{\{\gamma_i\}}$ for $\gamma_1,\gamma_2\in \Gamma_0$, and the intersection is $\{1_{\{\gamma_1,\gamma_2\}}\}$,
\item  or  we can assume that $a_1=E_{\gamma_0}$ for $\gamma_0\in \Gamma_0$ and $a_2=E_t$ where
$\gamma_0\in t$ and $t\not\sub \Gamma_0$, in which case the intersection is $a_2= \{1_t\}$.
\end{enumerate}

 It follows that whenever $\cF\sub\cG$ is finite then every vertex $\cY$ in $\graph(\cF)$ satisfies $|\cY|\le 2$ and $\cY$ is connected to the vertex $\emptyset\sub\cF$
 by a path in $\graph(\cF)$.
 \end{proof}

We do not know if   the above statement remains true  for subsets $\sigma_3(\Gamma)$.
Our candidate for a  counterexample is a subspace $L$ of $\sigma_3(\er)$, of the family of characteristic functions of all sets of the form $\{x,x+r,x+2r\}$, where $x,r\in\er$, together with the empty set. Every element of $L$ is then an arithmetic triple or a singleton, and
$L$ is clearly compact. It is easy to see that the obvious family of clopens $\{E_{\{x\}} : x\in\er\}$
is not isomorphic to a family of pseudoclopens in a connected compactum.

%However, the following observation shows that $L$ is a Boolean image
%of a compact connected space.
%
%\begin{remark}
%Suppose that $K\sub\sigma_3(\Gamma)$ is a compact space and there is $\Gamma_1\sub \Gamma$ such that $|\Gamma_1|=|\Gamma|$
%and $K$ contains $\sigma_2(\Gamma_1)$. Fix an injection $\vf: \Gamma\sm \Gamma_1\To \Gamma_1$.
%Then the family
%$$\{E_\gamma: \gamma\in\Gamma_1\}\cup\{ E_\gamma\cup E_{\vf(\gamma)}: \gamma\in\Gamma\sm\Gamma_1\},$$
%separates the points of $K$ and is isomorphic to a family of pseudoclopens in a connected compactum.
%\end{remark}

\section{Boolean images of separably connected spaces}\label{secsepcon}

In this section we give a proof of  Theorem \ref{alternative}, which is based on the  auxiliary result given below. We shall use here the
following notation: Given a family $\cG$ in $\clop(L)$ and $x,y\in L$, we write $\cG(x,y)$ for the family of those $a\in\cG$ which separate $x$ and $y$,
that is $|a\cap\{x,y\}|=1$.

\begin{thm}\label{decxy}
Let $K$ be a separably connected compact space, and let $\mathcal{G}$ be a family of clopen subsets of $L$ which is isomorphic to a family of pseudoclopens of $K$.

Then for every $x,y\in L$ there is a countable decomposition $\mathcal{G} = \bigcup_{n<\omega}\mathcal{G}_n$
such that for every $n$  and  every partition $\mathcal{G}_n = \mathcal{F}\cup \mathcal{H}$, if
$$\mathcal{F}\supset \{a\in\mathcal{G}_n : x,y\in a\} \mbox{  and }\mathcal{H}\supset\{a\in \mathcal{G}_n : x,y\not\in a\},$$
then $\bigcap\mathcal{F} \setminus \bigcup \mathcal{H}\neq\emptyset$.
In particular,
the family $\cG_n(x,y)$ is independent for every $n<\omega$.

\end{thm}

\begin{proof}
Let $\phi:\mathcal{G}\To \tilde{\mathcal{G}}$ be an isomorphism of $\mathcal{G}$ with a family of pseudoclopens of $K$. Pick
$$x' \in \bigcap\{\phi(a)^- : a\in \mathcal{G}, x\in a\}\setminus \bigcup\{\phi(a)^+ : a\in \mathcal{G}, x\not\in a\},$$
$$y' \in \bigcap\{\phi(a)^- : a\in \mathcal{G}, y\in a\}\setminus \bigcup\{\phi(a)^+ : a\in \mathcal{G}, y\not\in a\}.$$
The fact that the two  sets appearing are nonempty follows immediately by the definition of an isomorphism and compactness.

Let $S $ be a separable connected subspace of $K$ that contains $x'$ and $y'$ and  let $D = \{d_n : n<\omega\}$ be a countable dense subset of $S $.
We consider a countable decomposition $\mathcal{G} = \bigcup_{n<\omega} \mathcal{G}_n$, where $\cG_n$ is defined by the condition
$$d_n\in S \cap \left( \phi(a)^+ \setminus \phi(a)^-\right) \mbox{ whenever } a\in \mathcal{G}_n \mbox{ and }S \cap \left( \phi(a)^+ \setminus \phi(a)^-\right)\neq\emptyset.$$
We fix now $n<\omega$, and take a partition  $\mathcal{F}, \mathcal{H}$ of $\mathcal{G}_n$
%We are to check that $\bigcap\mathcal{F} \setminus \bigcup \mathcal{H}\neq\emptyset$.
 such that
$$\mathcal{F}\supset \{a\in\mathcal{G}_n : x,y\in a\} \mbox{ and } \mathcal{H}\supset\{a\in \mathcal{G}_n : x,y\not\in a\}.$$
We claim that
$$(\star)\ \ d_n\in \bigcap\{\phi(a)^+ : a\in \mathcal{F}\} \setminus \bigcup \{\phi(a)^- : a\in \mathcal{H}\}.$$
To verify the claim, consider $a\in\cG_n$ and the following two cases.

Suppose that  $\left( \phi(a)^+ \sm \phi(a)^-\right)  \cap S\neq \emptyset$.  Then $d_n\in \phi(a)^+ \sm \phi(a)^-$ by the definition of $\cG_n$,
so, for instance,  $d_n\in \phi(a)^+$ no matter if $a_n$ is in $\cF$ or in $\cH$. Similarly,   $d_n\notin\phi(a)^-$.

Suppose now that  $\left( \phi(a)^+ \sm \phi(a)^-\right) \cap S = \emptyset$.
Since $S $ is connected, either $S \subset \phi(a)^-$ or $S  \subset K\setminus \phi(a)^+$.
If the former holds then $x',y' \in \phi(a)^-$, which implies that $x,y \in a$ (by the definition of $x',y'$).
 We conclude that $a\in \cF$ (by our hypothesis on $\cF, \cH$) and then we get that
 $d_n \in S \subset \phi(a)^- \subset \phi(a)^+$, as desired.
 If the latter occurs, i.e.\ $S\subset K\setminus \phi(a)^+$ then $x',y'\not\in \phi(a)^+$,  which implies that $x,y\not\in a$
 (by the definition of $x',y'$), which in turn implies that $a\in \cH$ (by our hypothesis on $\cF, \cH$) and then
 $$d_n\in S \subset K\setminus \phi(a)^+ .$$
This completes the proof of $(\star)$.
\medskip

 Now using the fact that   $\phi$ is an isomorphism  and $(\star)$ we deduce that $\bigcap\mathcal{F}\setminus \bigcup \mathcal{H} \neq\emptyset$, as required.
\end{proof}

As a corollary we obtain Theorem~\ref{alternative}, this is its proof:

\begin{proof}
Let $\mathcal{G}$ be a family of clopen subsets of $L$  that separates the points of $L$,   which is isomorphic to a family of pseudoclopens of $K$. Let $1_a$ denote the characteristic function of a set $a$, so that $1_a(x) = 1$ if $a\in x$ and $1_a(x) = 0$ if $x\not\in a$. We consider two cases.
\begin{enumerate}[(i)]
\item Suppose that for every $x,y\in L$, the family  $\cG(x,y)$ is countable. In this case, $L$ is Corson, because for a fixed $x\in L$, the map $y\mapsto (1_a(y) - 1_a(x))_{a\in\mathcal{G}}$ embeds $L$ into $\Sigma(\mathbb{R}^\mathcal{G})$.
\item In the remaining case there are  $x,y\in L$ such that $\cG(x,y)$ is uncountable. Consider then the decomposition $\mathcal{G} = \bigcup_{n<\omega}\mathcal{G}_n$ given by Theorem \ref{decxy}. Then  $\cY = \cG(x,y)\cap \mathcal{G}_n$ is an uncountable independent family for some $n$ so
    the mapping $f:L \To 2^\cY$ given by $f(x) = (1_a(x))_{a\in \cY}$ is a continuous surjection.
\end{enumerate}
The proof is complete.
\end{proof}

Let us now explore some consequences of Theorem~\ref{alternative}.

\begin{cor}
If $K$ is separably connected space that does not map onto $[0,1]^{\omega_1}$, then every Boolean image of $K$ is Corson compact.
\end{cor}

\begin{proof}
Otherwise, $K$ has a Boolean image that maps onto $2^{\omega_1}$. By Proposition~\ref{cios}, this implies that $K$ has a closed subspace that maps continuously onto $2^{\omega_1}$, hence also onto $[0,1]^{\omega_1}$. By Tietze's extension theorem, $K$ itself maps continuously onto $[0,1]^{\omega_1}$.
\end{proof}

\begin{cor}
Let $K$ be a  connected compactum that is separable and does not map continuously onto $[0,1]^{\omega_1}$. Then every Boolean image of $K$ is metrizable.
\end{cor}

\begin{proof}
Let $L$ be a Boolean image of $K$. By Proposition \ref{cios}, $L$ is a continuous image of some closed subspace $K_0\sub K$.
Note that a continuous surjection $g:L\To 2^{\omega_1}$ would give a continuous surjection $K_0\To 2^{\omega_1}$, and consequently,
a continuous surjection $K\To [0,1]^{\omega_1}$, which is excluded by the assumption.
By Theorem~\ref{alternative}, $L$ is Corson compact.

On the other hand,  $K$ is separable, so it is a continuous image of $\beta\omega$. Hence by Proposition~\ref{preimage}, $L$ is separable.
Every separable Corson compact space is metrizable, and we are done.
\end{proof}

\begin{remark}
The spaces such as the split interval or $[0,\omega_1]$ are natural examples of zero-dimensional compacta that are Boolean images of
a compact connected spaces (by Corollary \ref{wn:4})
but are not Boolean images of compact separably connected spaces.
\end{remark}

Theorem~\ref{alternative} does not provide a definite answer about which zero-dimensional compact spaces are Boolean images of separably connected or path-connected compact sets. It is even unclear which Corson compact spaces are Boolean images of such kind. Let us consider a few more examples to add to the picture.

An adequate family is a family $\mathcal{A}$ of subsets of a set $\Gamma$ such that a set $A$ belongs to $\mathcal{A}$ if and only if all finite subsets of $A$ belong to $\mathcal{A}$. Such a family may be seen as  a compact space  $L=\{x\in 2^\Gamma: \{\gamma:x_\gamma=1\}\in\cA\}$.

\begin{prop}
Every space $L$ defined by an adequate family is a Boolean image of a path-connected compact space.
\end{prop}

\begin{proof}
Let $\mathcal{A}$ be an adequate family of subsets of $\Gamma$ and let
 $$L=\{x\in 2^\Gamma: \{\gamma:x_\gamma=1\}\in\cA\}.$$
  Consider the space $K$ defined by
$$K = \{x\in [0,1]^\Gamma : \{\gamma : x_\gamma\neq 0\}\in \mathcal{A}\}.$$
Note that $K$ is closed in $[0,1]^\Gamma$, for if $x\in [0,1]^\Gamma\sm K$ then  $S=\{\gamma : x_\gamma\neq 0\}\notin \mathcal{A}$
so $I\notin\mathcal{A}$ for some finite $I\sub S$, and  it is easy to define a neighbourhood of $x$ that is disjoint from $\mathcal{A}$.
Note that $L\sub K$ and $K$ is path-connected since every $x\in K$ is connected by a segment inside $K$ joining $x$ and $0\in K$.

Let $E_\gamma = \{x\in L: x_\gamma=1\}$ be the basic clopen subset of $L$, and
let $V_\gamma= \{x\in K : x_\gamma>0\}$.
It is easy to check that letting $\phi(E_\gamma) = (E_\gamma, V_\gamma)$
we define an isomorphism between $\{E_\gamma: \gamma\in\Gamma\}$ and the family of pseudoclopens in $K$.
\end{proof}

Another interesting example are tree spaces. Recall that a tree is a partially ordered set $T$ with a minimum element (its root) and in which each initial segment
$\{t\in T : t<s\}$ is well ordered. Let $\bar{T}$ be a family of subsets of $T$ that are either initial segments or full branches in $T$. Then
$\bar{T}$ is a again a tree when ordered by inclusion, it is in a sense the completion of $T$.
The tree space associated to $T$ is the compact subspace of $ 2^T$  made of the characteristic functions of elements of  $\bar{T}$.

\begin{prop}\label{treepath}
Let $T$ be a tree and let $L\sub 2^T$ be the tree space defined by $T$, that is for every $x\in L$ we have $\{t\in T: x_t=1\}\in \bar{T}$.
Then $L$  is a Boolean image of a compact connected space. If all branches of $T$ are countable, then $L$ is a Boolean image of a path-connected space.
\end{prop}

\begin{proof}
For every $t\in T$ let $F_t=\{x\in L: x_t=1\}$. Then $\{F_t:t\in T\}$ is a point-separating family of clopen subsets of  $L$.
Since  sets $F_t$ are either comparable or disjoint, the first assertion follows immediately from Corollary \ref{wn:3}.
We may, however, construct the desired connected compactum $K$ directly as follows.

Consider $K\subset [0,1]^T$ the set of all  $z$ that either belong to $L$ or there exists $t\in T$ such that $z_s = 1$ if $s<t$, and $z_s = 0$ if $s\not\leq t$.
Note that  in such a case,  $z_t$ is an arbitrary element of $[0,1]$.
This means that we are creating a path to connect each successor segment to its predecessor. This easily implies that the space $K$ is connected;  it is moreover
path-connected if all branches of $T$ are countable.

For every $t\in T$ put $V_t=\{x\in K: x_t>0\}$. Then $\phi(F_t) = (F_t, V_t)$ defines the required  isomorphism. Indeed,
if
$$\bigcap_{t\in I} V_t\sm \bigcup_{s\in J} F_s\neq\emptyset,$$
for some finite $I,J\sub T$ then

\begin{enumerate}[--]
\item $I$ consists of elements that are comparable in $T$ so $\bigcap_{t\in I} V_t=V_{t_0}$, where $t_0$ is the maximal element of $I$;
\item no element of $J$ is below $t_0$.
\end{enumerate}

Hence, taking $x\in L$ such that $x_t=1$ for $t<t_0$ and $x_t=0$ otherwise, we have $x\in \bigcap_{t\in I} F_t\sm \bigcup_{s\in J} F_s$.
\end{proof}

\section{Boolean images of convex spaces}\label{secconvex}

The result given below is an augmented version of Theorem \ref{decxy}.

\begin{thm}\label{decconvex}
Suppose that $K$ is a convex compact subset of $[0,1]^\Gamma$. Let $\mathcal{G}$ be a family of clopen subsets of $L$ which is isomorphic to a family of pseudoclopens of $K$. Then there is a countable decomposition $\mathcal{G} = \bigcup_{n<\omega}\mathcal{G}_n$ such that for every $x,y\in L$ and for every $n<\omega$ there exists a further finite decomposition $\mathcal{G}_n = \bigcup_{m=0}^n \mathcal{G}_{nm}$ such that for every partition $\mathcal{G}_{nm} = \mathcal{F}\cup \mathcal{H}$, if $$\mathcal{F}\supset \{a\in\mathcal{G}_{nm} : x,y\in a\} \mbox{ and }\mathcal{H}\supset\{a\in \mathcal{G}_{nm} : x,y\not\in a\}$$
 then $\bigcap\mathcal{F} \setminus \bigcup \mathcal{H}\neq\emptyset$.
 In particular,
the family $\cG_{nm}(x,y)$, of those elements of $\cG_{nm}$ that separate $x$ and $y$,  is independent for every $m<\omega$.

\end{thm}

\begin{proof}
Let $\phi:\mathcal{G}\To \tilde{\mathcal{G}}$ be an isomorphism of $\mathcal{G}$ with a family of pseudoclopens of $K$. For $x,y \in [0,1]^\Gamma$, let
$\rho(x,y) = \sup_{\gamma\in\Gamma} |x_\gamma - y_\gamma|$. Note that if $F,H$ are disjoint closed sets in $K$, then
$$\rho(F,H) = \inf\{\rho(x,y) : x\in F, y\in H\}>0.$$
Choose the decomposition $\mathcal{G} = \bigcup_n \mathcal{G}_n$ in such a way that $\rho(\phi(a)^-,K\setminus \phi(a)^+) > 1/n$ for all $a \in \mathcal{G}_n$. Now, fix $n<\omega$ and $x,y\in L$ and pick
$$x' \in \bigcap\{\phi(a)^- : a\in \mathcal{G}, x\in a\}\setminus \bigcup\{\phi(a)^+ : a\in \mathcal{G}, x\not\in a\},$$
$$y' \in \bigcap\{\phi(a)^- : a\in \mathcal{G}, y\in a\}\setminus \bigcup\{\phi(a)^+ : a\in \mathcal{G}, y\not\in a\}.$$
Let $D = \{d_m : m=0,\ldots,n\}$ be a $1/n$-dense subset of the segment
$$S=\{r\cdot x'+(1-r)\cdot y': r\in [0,1],$$
joining $x'$ and $y'$, in the sense that for every $z\in S$ there exists $d_m\in D$ such that $\rho(z,d_m)< 1/n$; we can suppose that $x',y'\in D$.

Consider families $\cG_{nm}$  defined by the condition
$$d_m\in S\cap \left( \phi(a)^+ \setminus \phi(a)^-\right) \mbox{ whenever } a\in \mathcal{G}_{nm} \mbox{ and } S\cap \left(\phi(a)^+ \setminus \phi(a)^-\right) \neq\emptyset.$$
Then $\mathcal{G}_n = \bigcup_{m=1}^{n} \mathcal{G}_{nm}$
because if $a\in \mathcal{G}_n$ and  the set $S\cap \left(\phi(a)^+ \setminus \phi(a)^-\right)$  is nonempty  then it must contain an subinterval of
$S$ which is either initial or final or of length greater than $1/n$. We fix $m\leq n$, and a partition
$\mathcal{G}_{nm} = \mathcal{F}\cup \mathcal{H}$, such that $\mathcal{F}\supset \{a\in\mathcal{G}_{nm} : x,y\in a\}$ and $\mathcal{H}\supset\{a\in \mathcal{G}_{nm} : x,y\not\in a\}$.  Then,
$$(\star)\ \ d_m\in \bigcap\{\phi(a)^+ : a\in \mathcal{F}\} \setminus \bigcup \{\phi(a)^- : a\in \mathcal{H}\},$$
which can be checked following the argument appearing the the proof of Theorem \ref{decxy}.
Likewise, using compactness and the fact that $\phi$ is an isomorphism, from $(\star)$ we deduce that $\mathcal{F}\setminus \mathcal{H} \neq\emptyset$.
\end{proof}

With the help of Theorem~\ref{decconvex} we shall give an example of a zero-dimensional space which is the Boolean image of a path-connected space but
cannot be a Boolean image  of  a convex space.

Let $L\sub 2^\Gamma$ be a compact space. Recall that every $a\in\clop(L)$ depends on a finite number of coordinates, that is there is a finite
set $F\sub\Gamma$ such that for every $x\in a$ and $y\in L$, if $x_\gamma=y_\gamma$ for all $\gamma\in F$ then $y\in a$.

The space we are going to discuss is defined by some tree $T$ (see Proposition \ref{treepath}); recall that we mean
$$L=\{x\in 2^T: \{t\in T: x_t=1\}\in \bar{T}\}\sub 2^T.$$

We need the following two general observations.

\begin{lemma}\label{5:2}
Let $L\sub 2^T$ be the space defined by a tree $T$.
Suppose that $(a_n)_n$ is a sequence in $\clop(L)$ such that every $a_n$ depends on coordinates in $F_n\sub T$, where
$|F_n|\le k$, and $k$ is  a fixed natural number.

Then the sequence $(a_n)_n$ is not independent.
\end{lemma}

\begin{proof}
For every $t\in T$ we define $x(t),y(t)\in L$, where $x(t)=1_{\{s: s<t\}}$, $y(t)=1_{\{s: s\leq t\}}$.
\medskip

\noindent {\sc Claim.} Let a set $a\in \clop(L)$ depends on coordinates in $F\sub T$.
If $a\neq\emptyset$ and $a\neq L$ then there is $t\in F$ such that either $x(t)\in a$ or $y(t)\in a$.
\medskip

To prove Claim suppose that the constant function $0\in L$ is not in $a$ and take any $x\in a$. Then for $B=\{t\in T: x_t=1\}\in\bar{T}$ we must have
$B\cap F\neq\emptyset$. Taking $t=\max(B\cap F)$ we get $y(t)\in a$ because $x\in a$ and the points $x$ and $y(t)$ agree on $F$.
In the remaining case, when $0\in a$, take $x\notin a$ and define $B$ as before. Then for $t=\min (B\cap F)$ we have
$x(t)\in a$.
\medskip

Take now $n$ such that $2^n>2nk$. Every atom $b$ of the finite sequence $a_1,\ldots, a_n$ depends on coordinates
in the set $F=\bigcup_{i\le n} F_i$, where $|F|\le k n$.  The set $\{x(t),y(t): t\in F\}$ has at most $2|F|\le 2kn$ points, so
Claim implies that some atom $b$ must be empty.
\end{proof}

 Recall that a subset of a tree is an antichain if it contains no pair of comparable elements in the tree order. A tree is called {\em special}
 if it is a union of countably many antichains.

\begin{lemma}\label{5:3}
Let $T$ be a tree admitting a decomposition $T=\bigcup_n T_n$, where, for every $n$, $T_n$ contains no infinite chain.
Then the tree $T$ is special.
\end{lemma}

\begin{proof}
Let $T_{nk}$ be the set of all $t\in T_n$ such that  $|\{s\in T_n: s<t\}|=k$.
Then $T=\bigcup_{n,k} T_{nk}$ and every $T_{nk}$ is clearly an antichain.
\end{proof}

The tree $T$ we use below consists of all the subsets of the set of rational numbers $\qu$ that are well-ordered by the usual order on the real line.
The tree order in $T$ is defined by end-extension, that is for $t,s\in T$, we define $t\leq s$  if $t$ is an initial segment of $s$.
Note that  every branch in $T$ is countable.
We shall use the fact that the tree $T$ is not special, see Todorcevic \cite[Corollary 9.9]{Todorders}.

\begin{thm}
Let $T$ be the tree of all well ordered subsets of $\mathbb{Q}$ ordered by end-extension. Then the space $L\sub 2^T$ defined by $T$ is a Boolean image of a path-connected space but it is not a Boolean image of any compact convex set.
\end{thm}

\begin{proof}
The first statement follows from Proposition~\ref{treepath}, so we focus on the second statement.

We argue by contradiction: Let $\mathcal{G}\sub\clop(L)$ be a point-separating family which is isomorphic to a family of pseudoclopens of a convex compact set, so that we have the decomposition $\mathcal{G} = \bigcup_n \mathcal{G}_n$ as in  Theorem~\ref{decconvex}. Given a clopen subset $a$ of $L$,
let $F(a)\subset T$ be the finite set of coordinates on which $a$ depends. For every $t\in T$ there must exist a clopen set $a_t\in \mathcal{G}$ that separates
$x(t)$ from $y(t)$, where $x(t)=1_{\{s: s<t\}}$, $y(t)=1_{\{s: s\leq t\}}$.
We shall consider also the finite set of rational numbers $Q_t$ defined by
$$Q_t = \{\min(t\setminus s) : s \in F(a_t), t\setminus s \neq \emptyset \}$$

For every $n,k<\omega$, finite set $Q\subset \mathbb{Q}$ and $\xi\in\{0,1\}$, let
$$T(n,k,Q,\xi) = \{t\in T : a_t\in \mathcal{G}_n, |F(a_t)|=k, Q_t = Q, 1_{a_t}(x(t)) = \xi\}.$$
\medskip

\noindent {\sc Claim.} For fixed $n$, $Q$ and $\xi$, the set $T(n,k,Q,\xi)$ does not contain any infinite chain.
\medskip

Note first that Claim leads quickly to the desired  contradiction. Indeed, we have
$$T = \bigcup_{n,Q,\xi} T(n,Q,\xi),$$
and once we know that every $T(n,Q,\xi)$ contains no infinite chain,
we conclude form Lemma \ref{5:3} that the tree $T$ is special while, as we mentioned above, this is not the case (\cite[Corollary 9.9]{Todorders}).

Thus we need only to verify the claim above.
Suppose, on the contrary, that we had an infinite chain $t_0<t_1 < \cdots$ inside $T(n,k,Q,\xi)$.
Let $t_\infty = \bigcup_{i<\omega} t_i\in T$; let $q_i = \min(t_{i+1}\setminus t_i)$ and $q_\infty = \sup_i q_i$ (so that
$q_\infty$ is either a real number or $q_\infty=+\infty$).

By removing a finite number of  elements of the chain, we can suppose that $[q_0,q_\infty)\cap Q = \emptyset$.
This implies that, for every $i<\omega$, the following holds
$$(\star) \ \ \left\{s\in F(a_{t_i}) : t_0 \leq s \leq t_\infty\right\} = \{t_i\}.$$
Indeed, note first that $t_i\in F(a_{t_i})$ because $a_{t_i}$ separates $x(t_i)$ from $y(t_i)$.
If we had $s\in F(a_{t_i})$ with $t_0\leq s\leq p_\infty$ and $s\neq t_i$ then, for instance $s<t_i$ and $q=\min(t_i\sm s) \in Q$
is in $[q_0,q_\infty)$, but this was excluded.

For simplicity in notation, we suppose that $\xi=0$, the other case being completely analogous.
In other words, we assume that  $x(t_i)\not\in a_{t_i}$ but $y(t_i)\in a_{t_i}$ for every $i$.

Consider now $x=x(t_0)$ and $z=x(t_\infty)$.
The condition $(\star)$ above implies that $x\notin a_i$ and $z\in a_i$ for every $i$ (by the fact that $a_i$ depends on coordinates in $F(a_i)$).
We take the finite decomposition $\mathcal{G}_n = \bigcup_m \mathcal{G}_{nm}$ of Theorem~\ref{decconvex} associated to this $x$ and $z$.
By passing to a subchain, we can suppose that all $a_{t_i}$ belong to the same piece $\mathcal{G}_{nm}$ of the partition
and by Theorem~\ref{decconvex}, the sequence of $a_{t_i}$ is independent. But this is impossible in view of
Lemma \ref{5:2}. This shows Claim and the proof is complete.
\end{proof}

Finally, we give a result on some class of Boolean images of convex compacta.
Given a natural number $n$, an adequate family $\mathcal{A}$ of sets  is $n$-adequate if a set belongs to $\mathcal{A}$ if and only if all of its subsets of cardinality at most $n$ belong to $\mathcal{A}$.

\begin{prop}\label{5.5}
Every compact space $L$ defined by an $n$-adequate family for some $n$ is a Boolean image of a compact convex set.
\end{prop}

\begin{proof}
Let $\cA$ be an $n$-adequate family of subsets of some $\Gamma$. Recall that
$$L= \{x\in 2^\Gamma: \{\gamma:x_\gamma=1\}\in\cA\}.$$
We define  $K$ to be the closed convex hull of $L$ inside $[0,1]^\Gamma$.
Writing
$$E_\gamma=\{x\in L: x_\gamma=1\},  V_\gamma=\{x\in K: x_\gamma>1-1/n\},$$
we check that the formula  $\phi(E_\gamma) = (E_\gamma, V_\gamma)$ gives the required isomorphism.

\noindent{\sc Claim.}
If $x\in K$ then $A=\{\gamma: x_\gamma>1-1/n\}\in \cA$.
\medskip

Indeed, if $A\notin\cA$ then take $I\sub A$ of size $\le n$ such that $I\notin\cA$. On the other hand, consider any  convex combination
$y=\sum_{k\le m} r_k\cdot x^k$, where $x^1,\ldots x^m\in L$. Then for every $k$ there is $\gamma_k\in I$ such that  $x^k_{\gamma_k}=0$.
Then there is $\gamma\in I$ such that
$$\sum_{\gamma_k=\gamma} r_k\ge 1/n,$$
(since $|I|\le n$). This gives $y_\gamma\le 1-1/n$.
It follows that $x$ is not in the closure of ${\rm conv} (L)$, a contradiction.
\medskip

Checking that $\phi$ is an isomorphism reduces easily  to verifying that for any finite $A\sub \Gamma$,
$$\mbox{ if } \bigcap_{\gamma\in A} V_\gamma\neq\emptyset \mbox{ then }   \bigcap_{\gamma\in A} E_\gamma\neq\emptyset,$$
but this follows directly from  Claim.
\end{proof}

\end{document}